\documentclass{article}
\usepackage{amsmath,amsthm,amssymb,mathrsfs,amstext, titlesec,enumitem}
\makeatletter

\titleformat{\section}
{\normalfont\Large\bfseries}{\thesection}{1em}{}
\titleformat*{\subsection}{\Large\bfseries}
\titleformat*{\subsubsection}{\large\bfseries}
\titleformat*{\paragraph}{\large\bfseries}
\titleformat*{\subparagraph}{\large\bfseries}

\renewcommand{\@seccntformat}[1]{\csname the#1\endcsname. }

\renewenvironment{abstract}{%
    \if@twocolumn
      \section*{\abstractname}%
    \else 
      \begin{center}%
        {\bfseries \Large\abstractname\vspace{\z@}}
      \end{center}%
      \quotation
    \fi}
    {\if@twocolumn\else\endquotation\fi}
    
\makeatother
\theoremstyle{plain}
\newtheorem{thm}{Theorem}[section]
\newtheorem{lem}[thm]{Lemma}
\theoremstyle{definition}
\newtheorem{defn}[thm]{Definition}

\newtheorem{qn}[thm]{Question}
\newtheorem{ex}[thm]{Example}
\providecommand{\keywords}[1]{{\bf{Keywords:}} #1}
\providecommand{\subjectclass}[2]{\textbf{Mathematics subject classification 2020:} #1}

\title{Dynamical IP$^{\star}$-sets in weak rings}

\date{\today}
\author{Pintu Debnath
\footnote{Department of Mathematics, Basirhat College, Basirhat -743412, North
24th parganas, West Bengal, India.
           {\bf pintumath1989@gmail.com}}
           \and
Sayan Goswami
\footnote{Department of Mathematics, 
          University of Kalyani, 
          Kalyani-741235,
          Nadia, West Bengal, India
          {\bf sayan92m@gmail.com}}
}

\date{\vspace{-5ex}}

\begin{document}
\maketitle

\begin{abstract}
\noindent V. Bergelson and N. Hindman in \cite{key-1}, proved
that IP$^{\star}$- sets contain all possible finite sums and products
of a sum subsystem of any sequence in $\mathbb{N}$. In a recent work
\cite{key-4}, the second author of this article has proved that a stronger result holds for dynamical IP$^{\star}$- sets. In this
article we will establish a non-commutative version of this result.
We will prove that a richer configuration is contained in dynamical
IP$^{\star}$- sets in weak rings.
\end{abstract}
\subjectclass{05D10}\\

\noindent \keywords{IP$^{\star}$-sets, Zigzag configuration, weak ring}

\section{Introduction}

Ramsey theoretic study is deeply involved with the study of topological
dynamics, ergodic theory and the algebra of the Stone-\v{C}ech compactification
of discrete semigroups. In \cite{key-4-1}, N. Hindman established
a result regarding partition regularity of finite sums of a sequence.
This theorem has motivated many mathematicians to study the IP-sets in Ramsey theory. For any set $X$, denote by $\mathcal{P}_{f}\left(X\right)$, the set of all nonempty finite subsets of $X$. Let
$\left(S,\cdot\right)$ be a semigroup and a set $A\subseteq S$ is
called an IP-set if and only if there exists a sequence $\langle x_{n}\rangle_{n=1}^{\infty}$
in $S$ such that $FP\left(\langle x_{n}\rangle_{n=1}^{\infty}\right)\subseteq A$.
The expression $FP\left(\langle x_{n}\rangle_{n=1}^{\infty}\right)$ denotes the set $\left\{ \prod_{n\in F}x_{n}:F\in\mathcal{P}_{f}\left(\mathbb{N}\right)\right\} $
and $\prod_{n\in F}x_{n}$ is defined to be the product in increasing order.
The following theorem is known as the Hindman's Theorem.
\begin{thm}
\rm{\cite{key-4-1}} Let $r\in\mathbb{N}$ and $\mathbb{N}=C_{1}\cup C_{2}\cup\ldots\cup C_{r}$
be a finite partition of $\mathbb{N}$. Then there exist $i\in\left\{ 1,2,\ldots,r\right\} $
and a sequence $\langle x_{n}\rangle_{n=1}^{\infty}$ such that $FS\left(\langle x_{n}\rangle_{n=1}^{\infty}\right)\subseteq C_{i}.$
\end{thm}

One can show that the above theorem is equivalent to the statement
that if an IP-set is partitioned into finitely many cells, then one
of the cell is itself an IP-set. The above theorem is true for arbitrary
discrete semigroups.

An IP$^{\star}$-set is a subset of $S$ which intersects all possible
IP-sets of $S$. Given a sequence $\langle x_{n}\rangle_{n=1}^{\infty}$
in $S$, we say that $\langle y_{n}\rangle_{n=1}^{\infty}$ is a product
subsystem of $\langle x_{n}\rangle_{n=1}^{\infty}$ provided there
exists a sequence $\langle H_{n}\rangle_{n=1}^{\infty}$ of non-empty
finite subsets such that $\text{max}H_{n}<\text{min}H_{n+1}$ and $y_{n}=\prod_{t\in H_{n}}x_{t}$
for each $n\in\mathbb{N}$. Notice that if $\langle y_{n}\rangle_{n=1}^{\infty}$
is a product subsystem of $\langle x_{n}\rangle_{n=1}^{\infty}$,
then $FP\left(\langle y_{n}\rangle_{n=1}^{\infty}\right)\subseteq FP\left(\langle x_{n}\rangle_{n=1}^{\infty}\right)$.
In \cite{key-1}, the authors have proved that the IP$^{\star}$-sets
contain a rich combinatorial configuration by establishing the following
theorem.
\begin{thm}
Let $\langle x_{n}\rangle_{n=1}^{\infty}$ be a sequence in $\mathbb{N}$
and $A$ be an IP$^{\star}$-set in $\left(\mathbb{N},+\right)$.
Then there exists a sum subsystem $\langle y_{n}\rangle_{n=1}^{\infty}$
of $\langle x_{n}\rangle_{n=1}^{\infty}$ such that 
\[
FS\left(\langle y_{n}\rangle_{n=1}^{\infty}\right)\cup FP\left(\langle y_{n}\rangle_{n=1}^{\infty}\right)\subseteq A.
\]
\end{thm}

In \cite[Theorem 6]{key-4}, the second author of this article, has proved that a
 richer configuration is contained in certain IP$^{\star}$-sets. We will define
the zigzag finite product configuration for $l$ many sequences as
follows:
\begin{defn}
\cite{key-4}
\begin{enumerate}
\item For any $l\in\mathbb{N}$ and any $l$-sequences $\langle x_{1,n}\rangle_{n=1}^{\infty},$ $\langle x_{2,n}\rangle_{n=1}^{\infty}$
$,\ldots$ $,\langle x_{l,n}\rangle_{n=1}^{\infty}$ in $S$, define
the zigzag finite product
\[
ZFP\left(\langle\langle x_{i,n}\rangle_{n=1}^{\infty}\rangle_{i=1}^{l}\right)=\left\{ \begin{array}{c}
\prod_{t\in H}y_{t}:H\in\mathcal{P}_{f}\left(\mathbb{N}\right)\,\text{and}\\
y_{i}\in\left\{ x_{1,i},x_{2,i},\ldots,x_{l,i}\right\} \,\text{for}\,\text{any}\,i\in\mathbb{N}
\end{array}\right\} 
\]
\item For any $l,m\in\mathbb{N}$ and any $l$-sequences $\langle x_{1,n}\rangle_{n=1}^{\infty},$ $\langle x_{2,n}\rangle_{n=1}^{\infty}$
$,\ldots$ $,\langle x_{l,n}\rangle_{n=1}^{\infty}$ in $S$, define
\[
ZFP\left(\langle\langle x_{i,n}\rangle_{n=1}^{m}\rangle_{i=1}^{l}\right)=\left\{ \begin{array}{c}
\prod_{t\in H}y_{t}:H\subseteq\left\{ 1,2,\ldots,m\right\} \,\text{and}\,\\
y_{i}\in\left\{ x_{1,i},x_{2,i},\ldots,x_{l,i}\right\} \,\text{for}\,\text{any}\,i\in\mathbb{N}
\end{array}\right\} 
\]
\end{enumerate}
\end{defn}

Clearly for $l=1$, this is nothing but the ordinary finite product of
a sequence. In \cite{key-4}, the second author of this article asked
the following question.
\begin{qn}
Let $l\in\mathbb{N}$ and $A\subseteq\mathbb{N}$ be an IP$^{\star}$-
set in $\left(\mathbb{N},+\right)$. Then for any $l$ sequences  $\langle x_{1,n}\rangle_{n=1}^{\infty},$ $\langle x_{2,n}\rangle_{n=1}^{\infty}$ $,\ldots$ $,\langle x_{l,n}\rangle_{n=1}^{\infty}$
in $\mathbb{N}$ whether there exists a $l$ sum subsystems $\langle y_{i,n}\rangle_{n=1}^{\infty},$ of
$\langle x_{i,n}\rangle_{n=1}^{\infty},$ for each $i\in\left\{ 1,2,\ldots,l\right\} $
such that 
\[
ZFP\left(\langle\langle y_{i,n}\rangle_{n=1}^{\infty}\rangle_{i=1}^{l}\right)\bigcup ZFP\left(\langle\langle y_{i,n}\rangle_{n=1}^{\infty}\rangle_{i=1}^{l}\right) \subset A.
\]
\end{qn}

We don't know the answer of the above question. In \cite[Theorem 6]{key-4},
it has been established that certain IP$^{\star}$-sets, arise from the
recurrence of Topological dynamics, contain the above structure. 
To recall dynamical IP$^{\star}$-sets, we need the following definitions.
\begin{defn}
\cite[Definition 19.29, page 503]{key-5}
\begin{enumerate}
\item A measure space is a triple $\left(X,\mathcal{B},\mu\right)$ where
$X$ is a set, $\mathcal{B}$ is a $\sigma$-algebra of subsets of
$X$, and is a countably additive measure on $\mathcal{B}$ with $\mu\left(X\right)$
finite.
\item Given a measure space $\left(X,\mathcal{B},\mu\right)$ a function
$T:X\rightarrow X$ is a measure preserving transformation if and
only if for all $B\in\mathcal{B}$, $T^{-1}\left[B\right]\in\mathcal{B}$
and $\mu\left(T^{-1}\left[B\right]\right)=\mu\left(B\right)$.
\item Given a semigroup $S$ and a measure space $\left(X,\mathcal{B},\mu\right)$,
a measure preserving action of $S$ on $X$ is an indexed family $\langle T_{s}\rangle_{s\in S}$
such that each $T_{s}$ is a measure preserving transformation of
$X$ and $T_{s}\circ T_{t}=T_{st}$ . It is also required that if
$S$ has an identity $e$, then $T_{e}$ is the identity function
on $X$.
\item A measure preserving system is a quadruple $\left(X,\mathcal{B},\mu,\langle T_{s}\rangle_{s\in S}\right)$
such that $\left(X,\mathcal{B},\mu\right)$ is a measure space and
$\langle T_{s}\rangle_{s\in S}$ is a measure preserving action of
$S$ on $X$.
\end{enumerate}
\end{defn}

From \cite[Theorem 19.33, page 504]{key-5}, we get the following
theorem:
\begin{thm}
Let $\left(X,\mathcal{B},\mu,\langle T_{s}\rangle_{s\in S}\right)$
be a measure preserving system. Then for any $A\in\mathcal{B}$ with
$\mu\left(A\right)>0$, $C=\left\{ s\in S:\mu\left(A\cap T_{s}^{-1}[A]\right)>0\right\} $
an IP$^{\star}$- set.
\end{thm}

The following one is the definition of dynamical IP$^{\star}$-set. 

\begin{defn} \label{dynip} \cite[Definition 19.34, page 505]{key-5}
Let $S$ be a semigroup. A subset $C$ of $S$ is dynamical IP$^{\star}$-
set if and only if there exist a measure preserving system $\left(X,\mathcal{B},\mu,\langle T_{s}\rangle_{s\in S}\right)$
and an $A\in\mathscr{\mathcal{B}}$ with $\mu\left(A\right)>0$ such
that $\left\{ s\in S:\mu\left(A\cap T_{s}^{-1}[A]\right)>0\right\} \subseteq C$.
\end{defn}

In this article we will extend \cite[Theorem 6]{key-4} in non-commutative
setting.

\section{Dynamical IP$^{\star}$-sets in weak rings}

First we need to extend \cite[Theorem 19.35, page 505]{key-5} and
\cite[Lemma 7]{key-4} for arbitrary discrete semigroups in our purpose. The
proof of the following lemma is similar to the proof of \cite[Theorem 19.35]{key-5}.

For any semigroup $\left(S,\cdot\right)$, $s,t\in S$ and $A\subseteq S$, define
$s^{-1}A=\left\{ t\in S:st\in A\right\} $, $As^{-1}=\left\{ t\in S:ts\in A\right\} $
and $s^{-1}At^{-1}=\left\{ x\in S:sxt\in A\right\} $.
\begin{lem}
\label{Lemma 1} Let B be a dynamical IP$^{\star}$-set in $(S,\cdot)$.
 Then it follows that there is a dynamical IP$^{\star}$-set $C\subset B$ such that for
each $t\in C$, $t^{-1}C$ is a dynamical IP$^{\star}$-set
\end{lem}

\begin{proof}
Let us consider a probability space $\left(X,\mathcal{B},\mu\right)$
with a measure preserving action $\left(T_{s}\right)_{s\in S}$ of
$(S,\cdot)$ on $X$, a set $A\in\mathcal{B}$ such that $\mu\left(A\right)>0$
and
$\left\{ s\in S:\mu\left(A\cap T_{s}^{-1}A\right)>0\right\} \subseteq B$.
Let 
\[
C=\left\{ s\in S:\mu\left(A\cap T_{s}^{-1}A\right)>0\right\} .
\]
To see that $C$ is as required, let $t\in C$ and let $D=A\cap T_{t}^{-1}A$.
We claim that
\[
\left\{ s\in S:\mu\left(D\cap T_{s}^{-1}D\right)>0\right\} \subseteq t^{-1}C.
\]
Let $s\in S$ such that $\mu\left(D\cap T_{s}^{-1}D\right)>0$. Then
\[
D\cap T_{s}^{-1}D=A\cap T_{t}^{-1}A\cap T_{s}^{-1}\left(A\cap T_{t}^{-1}A\right)
\]
\[
\quad \quad \quad \subseteq A\cap T_{s}^{-1}\left(T_{t}^{-1}A\right)
\]

\[
=A\cap T_{ts}^{-1}A
\]

so $ts\in C$ and $s\in t^{-1}C.$
\end{proof}
The following lemma says that the intersection of finite numbers of dynamical
IP$^{\star}$-sets is dynamical IP$^{\star}$-set.
\begin{lem}
\label{lemma 2} Let $\left(S,\cdot\right)$ be a semigroup. Let $A,B\subseteq S$
be two dynamical IP$^{\star}$-sets in $\left(S,\cdot\right)$. Then it follows that
$A\cap B$ is also a dynamical IP$^{\star}$-set in $\left(S,\cdot\right)$.
\end{lem}

\begin{proof}
Let $\left(X,\mathcal{B},\mu,\langle T_{s}\rangle_{s\in S}\right)$
and $\left(Y,\mathcal{C},\nu,\langle R_{s}\rangle_{s\in S}\right)$
be two measure preserving systems and $C,D$ be two sets guaranteed
by the definition \ref{dynip}.
 Now we have $\left\{ s\in S:\mu\left(C\cap T_{s}^{-1}[C]\right)>0\right\} $ $\subseteq A$
and $\left\{ s\in S:\nu\left(D\cap R_{s}^{-1}[D]\right)>0\right\} \subseteq B$.
 Let $\left(X\times Y,\mathcal{D},\mu \otimes \nu,\langle T_{s}\times R_{s}\rangle_{s\in S}\right)$
be a dynamical system where $\mathcal{D}$ is the smallest $\sigma$-algebra of subsets of $X\times Y$  which contains $\left\{ B\times C:B\in\mathcal{B},C\in\mathcal{C}\right\}, $
 $\mu \otimes \nu$  is the countably additive measure on $\mathcal{D}$ such that $(\mu \otimes \nu)(\mathcal{B}\times\mathcal{C})=\mu (B)\cdot \nu (C)$ for each $B\in \mathcal{B}$ and $C\in \mathcal{C}$, and $(T_s \times R_s)(x,y)=\left(T_s(x),R_s(y)\right)$. 
So, 
\[
\left(\mu\otimes\nu\right)\left(\left(C\times D\right)\cap\left(T_{s}\times R_{s}\right)^{-1}[C\times D]\right)=\mu\left(C\cap T_{s}^{-1}[C]\right)\cdot\nu\left(D\cap R_{s}^{-1}[D]\right).
\]

This implies $s\in\left\{ \left(\mu\otimes\nu\right)\left(\left(C\times D\right)\cap\left(T_{s}\times R_{s}\right)^{-1}[C\times D]\right)>0\right\} $
iff $s\in\left\{ s:\mu\left(C\cap T_{s}^{-1}[C]\right)>0\right\} \cap\left\{ s:\nu\left(D\cap R_{s}^{-1}[D]\right)>0\right\} $
and so $A\cap B$ is dynamical IP$^{\star}$-set.
\end{proof}
Let us recall the definition of weak ring.
\begin{defn}
\cite[Definition 16.33, page 419]{key-5}
\begin{enumerate}
\item A left weak ring is a triple $\left(S,+,\cdot\right)$ such that $\left(S,+\right)$
and $\left(S,.\right)$ are semigroups and the left distributive law
holds. That is, for all $x,y,z\in S$ one has$x\cdot\left(y+z\right)=x\cdot y+x\cdot z$.
\item A right weak ring is a triple $\left(S,+,\cdot\right)$ such that
$\left(S,+\right)$ and $\left(S,.\right)$ are semigroups and the
right distributive law holds. That is, for all $x,y,z\in S$ one has$\left(x+y\right)\cdot z=x\cdot z+y\cdot z$.
\item A weak ring is a triple $\left(S,+,\cdot\right)$ which is both a
left weak ring and a right weak ring.
\end{enumerate}
\end{defn}

Recall that in $FP\left(\langle x_{n}\rangle_{n=1}^{\infty}\right)$
the products are taken in increasing order of indices. The following
definition of `all product' of a sequence is taken from \cite[Definition 16.36, page 420]{key-5}.
\begin{defn}
Let $\left(S,\cdot\right)$ be a semigroup, let $\langle x_{n}\rangle_{n=1}^{\infty}$
be a sequence in $S$, and let $k\in\mathbb{N}$. Then $AP\left(\langle x_{n}\rangle_{n=1}^{k}\right)$
is the set of all products of terms of $\langle x_{n}\rangle_{n=1}^{k}$in
any order with no repetitions. Similarly $AP\left(\langle x_{n}\rangle\right)_{n=1}^{\infty}$
is the set of all products of terms of $\langle x_{n}\rangle_{n=1}^{\infty}$
in any order with no repetitions.
\end{defn}

For example, for $k=3$, we obtain the following:

\[
AP\left(\langle x_{n}\rangle_{n=1}^{3}\right)=\{x_{1},x_{2},x_{3},x_{1}x_{2},x_{2}x_{1},x_{1}x_{3},x_3x_1,x_{2}x_{3},x_{3}x_{2},
\]
\[
\qquad\qquad\qquad\qquad\quad x_{1}x_{2}x_{3},x_{1}x_{3}x_{2},x_{2}x_{1}x_{3},x_{2}x_{3}x_{1},x_{3}x_{1}x_{2},x_{3}x_{2}x_{1}\}.
\]

From \cite[Theorem 16.38, page 421]{key-5}, we have the following
theorem for IP$^{\star}$-sets. 
\begin{thm}
Let be $\left(S,+,\cdot\right)$ be a weak ring, let $A$ be an IP$^{\star}$
set in $\left(S,+\right)$, and let $\langle x_{n}\rangle_{n=1}^{\infty}$
be any sequence in $S$. Then there exists a sum subsystem $\langle y_{n}\rangle_{n=1}^{\infty}$
of $\langle x_{n}\rangle_{n=1}^{\infty}$ in $S$ such that \textup{$FS(\langle x_{n}\rangle_{n=1}^{\infty})\cup AP(\langle x_{n}\rangle_{n=1}^{\infty})\subseteq A$.}
\end{thm}

In this article our aim is to prove a zigzag version
of the above theorem for dynamical IP$^{\star}$-sets.
We are thankful to the referee for improvement of the following definition from the previous draft.
\begin{defn}
Let $(S,+,\cdot)$ be a weak ring, let $l\in\mathbb{N}$, and let
$\langle x_{1,n}\rangle_{n=1}^{\infty}$, $\langle x_{2,n}\rangle_{n=1}^{\infty}$,
$\ldots$, $\langle x_{l,n}\rangle_{n=1}^{\infty}$, be any $l$ sequences
in $S$.
\begin{enumerate}
\item $ZAP\left(\langle\langle x_{i,n}\rangle_{n=1}^{\infty}\rangle_{i=1}^{l}\right)=\Big\lbrace\prod_{t=1}^{r}x_{\gamma\left(t\right),\tau\left(t\right)}:r\in\mathbb{N},\tau:\left\{ 1,2,\ldots,r\right\} \overset{1-1}{\longrightarrow}\mathbb{N},\,\text{and}\,\gamma:\left\{ 1,2,\ldots,r\right\} \rightarrow\left\{ 1,2,\ldots,l\right\} \Big\rbrace.$
\item For $k\in\mathbb{N}$, $ZAP\left(\langle\langle x_{i,n}\rangle_{n=1}^{k}\rangle_{i=1}^{l}\right)=\Big\lbrace\prod_{t=1}^{r}x_{\gamma\left(t\right),\tau\left(t\right)}:r\in\left\{ 1,2,\ldots,k\right\} ,\tau:\left\{ 1,2,\ldots,r\right\} \overset{1-1}{\longrightarrow}\left\{ 1,2,\ldots,k\right\} \,\text{and}\,\gamma:\left\{ 1,2,\ldots,r\right\} \rightarrow\left\{ 1,2,\ldots,l\right\} \Big\rbrace.$
\item $ZFS\left(\langle\langle x_{i,n}\rangle_{n=1}^{\infty}\rangle_{i=1}^{l}\right)=\!\!\Big\lbrace\!\sum x_{\gamma\left(t\right),t}:\!F\in\mathcal{P}_{f}\left(\mathbb{N}\right)\,\text{and}\,\gamma:F\rightarrow\left\{ 1,2,\ldots,l\right\} \Big\rbrace.$
\item For $k\in\mathbb{N},$ $ZFS\left(\langle\langle x_{i,n}\rangle_{n=1}^{k}\rangle_{i=1}^{l}\right)=\Big\lbrace\sum_{t\in F}x_{\gamma\left(t\right),t}:\emptyset\neq F\subseteq\left\{ 1,2,\ldots,k\right\}\linebreak  \,\text{and}  \,\gamma:F\rightarrow\left\{ 1,2,\ldots,l\right\} \Big\rbrace.$
\end{enumerate}
\end{defn}

Let us demonstrate the above definition by a simple example.
\begin{ex}
 Suppose we have $k,l=2$.
The following is the zigzag all product configuration:

$ZAP\left(\langle\langle x_{i,n}\rangle_{n=1}^{2}\rangle_{i=1}^{2}\right)=\Big\lbrace x_{1,1},x_{2,1},x_{1,2},x_{2,2},x_{1,1}x_{1,2},x_{1,2}x_{1,1},x_{1,1}x_{2,2},$

\hspace{1.55in} $ x_{2,2}x_{1,1}, x_{2,1}x_{1,2},x_{1,2}x_{2,1},x_{2,1}x_{2,2},x_{2,2}x_{2,1} \Big\rbrace $

\end{ex}

We need the following lemma.

\begin{lem}
Let $\left(S,+,\cdot\right)$ be a weak ring and let $A$ be a dynamical
$IP^{\star}$-set in $\left(S,+\right)$. Then for all $s,t\in S,$ $s^{-1}A,$ $As^{-1}$ and $s^{-1}At^{-1}$ are dynamical $IP^{\star}$-sets in $\left(S,+\right)$. 

\end{lem}

\begin{proof}
 Pick a measure preserving system $\left(X,\mathcal{B},\mu,\langle T_{t}\rangle_{t\in S}\right)$
and $C\in\mathcal{B}$ such that $\mu (C)> 0$ and $D=\left\{ v\in S:\mu\left(C\cap T_{v}^{-1}[C]\right)>0\right\} \subseteq A$. We will do the proof for $s^{-1}At^{-1}$. 

Let $s,t\in S$ be given and define for $u\in S,$ $R_u:X\rightarrow X $ by $R_u=T_{sut}.$ 
Then given $u,v \in S $, $ R_u \circ R_v=T_{sut} \circ T_{svt}=T_{sut+svt}=T_{s(u+v)t}=R_{u+v}.$ Therefore $\left(X,\mathcal{B},\mu,\langle R_{u}\rangle_{u\in S}\right)$ is a measure preserving system. It suffices to show that $\Big\lbrace u\in S: \mu (C\cap R_u^{-1}[C])> 0\Big\rbrace \subseteq s^{-1}At^{-1}.$ So assume that $u\in S$ and $\mu (C\cap R_u^{-1}[C])> 0$. Then $\mu (C\cap T_{sut}^{-1}[C])> 0$, so $sut\in D\subseteq A$ so $u\in s^{-1}At^{-1}.$
\end{proof}

\begin{lem} \label{newlem}
Let $\left(S,+\right)$ be a semigroup and let $A$ be an IP$^{\star}$-
set in $S$. Let $l\in\mathbb{N}$ and for $i\in\left\{ 1,2,\ldots,l\right\} $,
let $\langle x_{i,n}\rangle_{n=1}^{\infty}$ be a sequence in $S$.
There exists a sequence $\langle H_{n}\rangle_{n=1}^{\infty}$ in
$\mathcal{\mathcal{P}}_{f}\left(\mathbb{N}\right)$ such that for
each $n\in\mathbb{N}$, $\max H_{n}<\min H_{n+1}$ and, if for each
$n\in\mathbb{N}$ and $i\in\left\{ 1,2,\ldots,l\right\} $, $y_{i,n}=\sum_{t\in H_{n}}x_{i,t}$,
then for each $i\in\left\{ 1,2,\ldots,l\right\} $, $FS\left(\langle y_{i,n}\rangle_{n=1}^{\infty}\right)\subseteq A$.
In particular, there exists $H\in\mathcal{\mathcal{P}}_{f}\left(\mathbb{N}\right)$
such that for each $i\in\left\{ 1,2,\ldots,l\right\} $, $\sum_{t\in H_{n}}x_{i,t}\in A$.
\end{lem}
\begin{proof}
We proceed by induction on $l$. Assume first that $l=1$.
By \cite [Lemma 5.11]{key-5}, there exists an idempotent $p\in\beta S$
such that for each $m\in\mathbb{N}$, $FS\left(\langle x_{1,n}\rangle_{n=m}^{\infty}\right)\in p$.
Since $A$ is an IP$^{\star}$-set, $A\in p$. Then by  \cite[Theorem
5.14]{key-5}, there exists a sequence $\langle H_{n}\rangle_{n=1}^{\infty}$
in $\mathcal{\mathcal{P}}_{f}\left(\mathbb{N}\right)$ such that $\max H_{n}<\min H_{n+1}$
for each $n$ and, if $y_{1,n}=\sum_{t\in H_{n}}x_{1,t}$ then $FS\left(\langle y_{1,n}\rangle_{n=1}^{\infty}\right)\subseteq A$.

Now assume that $l>1$ and the statement is true for $l-1$.
Pick a sequence $\langle K_{n}\rangle_{n=1}^{\infty}$ in $\mathcal{\mathcal{P}}_{f}\left(\mathbb{N}\right)$
such that $\max K_{n}<\min K_{n+1}$ for each $n\in\mathbb{N}$ and
for each $i\in\left\{ 1,2,\ldots,l-1\right\} $, letting $z_{i,n}=\sum_{t\in K_{n}}x_{i,t}$
for each $n\in\mathbb{N}$, we have $FS\left(\langle z_{i,n}\rangle_{n=1}^{\infty}\right)\subseteq A$.
For each $n\in\mathbb{N}$, let $z_{l,n}=\sum_{t\in K_{n}}x_{l,t}$.
Using {[}5, Lemma 5.11 and Theorem 5.14{]} again, pick a sequence
$\langle M_{n}\rangle_{n=1}^{\infty}$ in $\mathcal{\mathcal{P}}_{f}\left(\mathbb{N}\right)$
such that $\max M_{n}<\min M_{n+1}$ for each $n$ and, if $y_{l,n}=\sum_{j\in M_{n}}z_{l,j}$,
then $FS\left(\langle y_{l,n}\rangle_{n=1}^{\infty}\right)\subseteq A$.

For each $n\in\mathbb{N}$, let $H_{n}=\bigcup_{j\in M_{n}}K_{j}$.
Note that for each $n$, $y_{l,n}=\sum_{j\in M_{n}}z_{l,j}=\sum_{j\in M_{n}}\sum_{t\in K_{j}}x_{l,t}=\sum_{t\in H_{n}}x_{l,t}$.
For $i\in\left\{ 1,2,\ldots,l-1\right\} $ and $n\in\mathbb{N}$,
let $y_{i,n}=\sum_{t\in H_{n}}x_{i,t}$. Then $y_{i,n}=\sum_{j\in M_{n}}z_{i,j}$
so $FS\left(\langle y_{i,n}\rangle_{n=1}^{\infty}\right)\subseteq FS\left(\langle z_{i,j}\rangle_{j=1}^{\infty}\right)\subseteq A$.
\end{proof}

We are thankful to the referee for his help to strengthen the following theorem  from our previous draft.
\begin{thm}
Let $\left(S,+,\cdot\right)$ be a weak ring and let $A$ be a dynamical
IP$^{\star}$- set in $\left(S,+\right)$. Let $l\in\mathbb{N}$ and
for $i\in\left\{ 1,2,\ldots,l\right\} $, let $\langle x_{i,n}\rangle_{n=1}^{\infty}$
be a sequence in $S$. There exists a sequence $\langle H_{n}\rangle_{n=1}^{\infty}$
in $\mathcal{\mathcal{P}}_{f}\left(\mathbb{N}\right)$ such that for
each $n\in\mathbb{N}$, $\max H_{n}<\min H_{n+1}$ and, if for each
$n\in\mathbb{N}$ and $i\in\left\{ 1,2,\ldots,l\right\} $, $y_{i,n}=\sum_{t\in H_{n}}x_{i,t}$,
$i\in\left\{ 1,2,\ldots,l\right\} $, $ZFS\left(\langle\langle y_{i,n}\rangle_{n=1}^{\infty}\rangle_{i=1}^{l}\right)\cup ZFP\left(\langle\langle y_{i,n}\rangle_{n=1}^{\infty}\rangle_{i=1}^{l}\right)\subseteq A$.
\end{thm}

\begin{proof}
By Lemma 2.1 we may pick a dynamical IP$^{\star}$-set $C\subseteq A$
such that for each $s\in C,$ $-s+C$ is dynamical IP$^{\star}$-set.

We construct $\langle H_{n}\rangle_{n=1}^{\infty}$ by induction so
that if $n>1$, then $\max H_{n}>\min H_{n-1}$ and, letting $y_{i,n}=\sum_{t\in H_{n}}x_{i,t}$,

$$ZFS\left(\langle\langle y_{i,t}\rangle_{t=1}^{n}\rangle_{i=1}^{l}\right) \cup ZFP\left(\langle\langle y_{i,t}\rangle_{t=1}^{n}\rangle_{i=1}^{l}\right)\subseteq C.$$

Pick by Lemma \ref{newlem}, $H_{1}\in\mathcal{\mathcal{P}}_{f}\left(\mathbb{N}\right)$
such that for each $i\in\left\{ 1,2,\ldots,l\right\} $, $y_{i,1}=\sum_{t\in H_{1}}x_{i,t}$,
then $y_{i,1}\in C$. Then 
\[
ZFS\left(\langle\langle y_{i,t}\rangle_{t=1}^{1}\rangle_{i=1}^{l}\right)\cup ZFP\left(\langle\langle y_{i,t}\rangle_{t=1}^{1}\rangle_{i=1}^{l}\right)=\left\{ y_{1,1},y_{2,1},\ldots,y_{l,1}\right\} \subseteq C
\]

Now let $n\in\mathbb{N}$ and assume that we have chosen $\langle H_{s}\rangle_{s=1}^{\infty}$
in $\mathcal{\mathcal{P}}_{f}\left(\mathbb{N}\right)$ such that $\min H_{n}>\max H_{n-1}$
if $n>1$ and $$B=ZFS\left(\langle\langle y_{i,s}\rangle_{s=1}^{n}\rangle_{i=1}^{l}\right)\cup ZFP\left(\langle\langle y_{i,s}\rangle_{s=1}^{n}\rangle_{i=1}^{l}\right)\subseteq C,$$
where for $s\in\left\{ 1,2,\ldots,n\right\} $ and $i\in\left\{ 1,2,\ldots,l\right\} $,
$y_{i,s}=\sum_{t\in H_{s}}x_{i,t}$ Let
\[
D=C\cap\bigcap_{s\in B}\left(-s+C\right)\cap\bigcap_{s\in B}s^{-1}C\cap\bigcap_{s\in B}Cs^{-1}\cap\bigcap_{s\in B}\bigcap_{u\in B}\left(s^{-1}Cu^{-1}\right).
\]

\noindent By Lemmas 2.2 and 2.8, $D$ is a dynamical IP$^{\star}$-set.

Let $p=\max H_{n}$. By Lemma \ref{newlem} applied to the sequences
$\langle x_{i,t}\rangle_{t=p+1}^{\infty}$ for $i\in\left\{ 1,2,\ldots,l\right\} $,
pick $H_{n+1}\in\mathcal{\mathcal{P}}_{f}\left(\mathbb{N}\right)$
with $\min H_{n+1}>p$ such that for each $i\in\left\{ 1,2,\ldots,l\right\} $,
$y_{i,n+1}\in D$ where $y_{i,n+1}=\sum_{t\in H_{n+1}}x_{i,t}$. We
need to show that $ZFS\left(\langle\langle y_{i,s}\rangle_{t=1}^{n+1}\rangle_{i=1}^{l}\right)\cup ZFP\left(\langle\langle y_{i,t}\rangle_{t=1}^{n+1}\rangle_{i=1}^{l}\right)\subseteq C$.

To see that $ZFS\left(\langle\langle y_{i,s}\rangle_{s=1}^{n+1}\rangle_{i=1}^{l}\right)\subseteq C$,
let $\emptyset\neq F\subseteq\left\{ 1,2,\ldots,n+1\right\} $, let
$\gamma:F\rightarrow\left\{ 1,2,\ldots,l\right\} $, and let $z=\sum_{t\in F}x_{\gamma\left(t\right),t}$.
If $n+1\notin F$, then $z\in C$ by assumption. If $F=\left\{ n+1\right\} $,
then $z=y_{\gamma\left(n+1\right),n+1}\in D\subseteq C$. So assume
that $n+1\in F$ and $G=F\setminus\left\{ n+1\right\} \neq\emptyset$.
Let $s=\sum_{t\in G}y_{\gamma\left(t\right),t}$. By assumption $s\in B$
so $y_{\gamma\left(n+1\right),n+1}\in-s+C$ so that $z\in C$.

 To see that $ZAP\left(\langle\langle y_{i,s}\rangle_{s=1}^{n+1}\rangle_{i=1}^{l}\right)\subseteq C$,
let us assume  $r\in\left\{ 1,2,\ldots,n+1\right\} $, let
$\tau:\left\lbrace 1,2,\ldots ,r\right\} \overset{1-1}{\longrightarrow}\left\{ 1,2,\ldots,n+1\right\rbrace $,
let $\gamma:\left\{ 1,2,\ldots,r\right\} \rightarrow\left\{ 1,2,\ldots,l\right\} $,
and let $z=\prod_{t=1}^{r}y_{\gamma\left(t\right),\tau\left(t\right)}$.
If $n+1\notin\tau\left[\left\{ 1,2,\ldots,r\right\} \right]$, then
$z\in C$ by assumption so assume we have $m\in\left\{ 1,2,\ldots,r\right\} $
such that $n+1=\tau\left(m\right)$. If $r=1$, then $m=1$ so that
$z=y_{\gamma\left(1\right),n+1}\in C$. So assume that $r>1$. If
$m=r,$ then $s=\prod_{t=1}^{r-1}y_{\gamma\left(t\right),\tau\left(t\right)}\in B$
so $y_{\gamma\left(r\right),n+1}\in s^{-1}C$ so $z\in C$.

If $m=1$, define 
$\eta:\left\{ 1,2,\ldots,r-1\right\}\overset{1-1}{\longrightarrow}\left\{ 1,2,\ldots,n\right\}$
and $\nu:\left\{ 1,2,\ldots,r-1\right\}$ $\rightarrow\left\{ 1,2,\ldots,l\right\}$
by $\eta\left(t\right)=\tau\left(t+1\right)$ and $\nu\left(t\right)=\gamma\left(t+1\right)$.
Then $s=\prod_{t=1}^{r-1}y_{\nu\left(t\right),\tau\left(t\right)}$ $\in B$
so $y_{\gamma\left(1\right),n+1}\in Cs^{-1}$ and thus $z=y_{\gamma\left(1\right),n+1}s\in C$.
Finally, assume $1<m<r$. Define $\eta:\left\{ 1,2,\ldots,r-m\right\} \overset{1-1}{\longrightarrow}\left\{ 1,2,\ldots,n\right\} $ and $\nu:\left\{ 1,2,\ldots,r-m\right\} \longrightarrow\left\{ 1,2,\ldots,l\right\} $
by $\eta\left(t\right)=\tau\left(m+t\right)$ and $\nu\left(t\right)=\gamma\left(m+t\right)$.
Then $s=\prod_{t=1}^{m-1}y_{\gamma\left(t\right),\tau\left(t\right)} \linebreak\in B$
and $u=\prod_{t=1}^{r-m}y_{\nu\left(t\right),\eta\left(t\right)}\in B$
so $z=sy_{\gamma\left(m\right),n+1}u\in C$.
\end{proof}

\vspace{0.2in}

\noindent \textbf{Acknowledgment: } The second author acknowledges the grant UGC-NET
SRF fellowship with id no. 421333 of the CSIR-UGC NET December 2016. We acknowledge the anonymous referee for his helpful  comments to improve the article.

\end{document}